\documentclass[10pt,reqno]{amsart}

\usepackage{color}

\newtheorem{thm}{Theorem}[section]

\newtheorem{corollary}[thm]{Corollary}
\newtheorem{lemma}[thm]{Lemma}

\newtheorem{assumption}[thm]{Assumption}

\theoremstyle{remark}
\newtheorem{remark}[thm]{Remark}

\def\Xint#1{\mathchoice
{\XXint\displaystyle\textstyle{#1}}%
{\XXint\textstyle\scriptstyle{#1}}%
{\XXint\scriptstyle\scriptscriptstyle{#1}}%
{\XXint\scriptscriptstyle\scriptscriptstyle{#1}}%
\!\int}
\def\XXint#1#2#3{{\setbox0=\hbox{$#1{#2#3}{\int}$}
\vcenter{\hbox{$#2#3$}}\kern-.5\wd0}}

\def\dashint{\Xint-}

\newcommand\cbrk{\text{$]$\kern-.15em$]$}}
\newcommand\opar{\text{\,\raise.2ex\hbox{${\scriptstyle
|}$}\kern-.34em$($}}
\newcommand\cpar{\text{$)$\kern-.34em\raise.2ex\hbox{${\scriptstyle |}$}}\,}

\def\<{\langle}
\def\>{\rangle}

\newcommand\bR{\mathbb{R}}

\newcommand\bM{\mathbb{M}}

\newcommand\bN{\mathbb{N}}

\newcommand\fR{\mathbf{R}}

\newcommand\cA{\mathcal{A}}

\newcommand\cF{\mathcal{F}}
\newcommand\cG{\mathcal{G}}

\newcommand\cI{\mathcal{I}}
\newcommand\cJ{\mathcal{J}}

\newcommand\cL{\mathcal{L}}

\newcommand\cM{\mathcal{M}}

\newcommand{\mysection}[1]{\section{#1}
\setcounter{equation}{0}}

\begin{document}

\title[BMO estimates for pseudo-differential operators]
{Parabolic BMO estimates  for  pseudo-differential operators of arbitrary order}

\author{Ildoo Kim}
\address{Department of Mathematics, Korea University, 1 Anam-dong, Sungbuk-gu, Seoul,
136-701, Republic of Korea} \email{waldoo@korea.ac.kr}

\author{Kyeong-Hun Kim}
\address{Department of Mathematics, Korea University, 1 Anam-dong,
Sungbuk-gu, Seoul, 136-701, Republic of Korea}
\email{kyeonghun@korea.ac.kr}
\thanks{This work was supported by Samsung Science  and Technology Foundation under Project Number SSTF-BA1401-02}

\author{Sungbin Lim}
\address{Department of Mathematics, Korea University, 1 Anam-dong, Sungbuk-gu, Seoul,
136-701, Republic of Korea} \email{sungbin@korea.ac.kr}


\begin{abstract}
In this article we prove the BMO-$L_{\infty}$ estimate
$$
\|(-\Delta)^{\gamma/2} u\|_{BMO(\fR^{d+1})}\leq N \|\frac{\partial}{\partial t}u-A(t)u\|_{L_{\infty}(\fR^{d+1})}, \quad \forall\, u\in C^{\infty}_c(\fR^{d+1})
$$
 for a wide class of pseudo-differential operators $A(t)$ of order $\gamma\in (0,\infty)$.
The coefficients of $A(t)$ are assumed to be merely measurable in time variable. As an application to the equation
$$
\frac{\partial}{\partial t}u=A(t)u+f,\quad t\in \fR
$$
we prove that for any $u\in C^{\infty}_c(\bR^{d+1})$
$$
\|u_t\|_{L_p(\fR^{d+1})}+\|(-\Delta)^{\gamma/2}u\|_{L_p(\fR^{d+1})}\leq N\|u_t-A(t)u\|_{L_p(\fR^{d+1})},
$$
where $p \in (1,\infty)$ and the constant $N$ is independent of $u$.
\end{abstract}

\maketitle

\mysection{Introduction}

It is a classical result that if a   second-order operator $A(t)u=a^{ij}(t)u_{x^ix^j}$ fulfills the uniform ellipticity
$$
\delta |\xi|^2 \leq a^{ij}(t)\xi^i\xi^j \leq \delta^{-1}|\xi|^2, \quad \delta>0
$$
 then it holds that for any $p>1$ and $u\in C^{\infty}_c(\fR^{d+1})$
\begin{align}
                \label{int 1}
\|\Delta u\|_{L_p(\fR^{d+1})} \leq c(\delta,p) \|u_t-A(t)u\|_{L_p(\fR^{d+1})}.
\end{align}
If $a^{ij}(t)$ are smooth  enough, then (\ref{int 1}) can be  obtained by using the  multiplier theory. The classical multiplier theory is not applicable if  $a^{ij}(t)$ are merely measurable in $t$. In this case one can rely on either Carlder\'on-Zygmund theory (see \cite{Kr01c}) or
the approach   based on the sharp  function estimate of  $\Delta u$ (see \cite{KrPbook}).

In this article we extend (\ref{int 1}) to a  wide  class of arbitrary order pseudo-differential operators $A(t)$ with measurable coefficients based on a BMO-$L_{\infty}$ estimate. More precisely we prove
\begin{equation}
                \label{real mainresult}
\|(-\Delta)^{\gamma/2}u\|_{BMO(\fR^{d+1})}\leq N \|u_t-A(t)u\|_{L_{\infty}(\fR^{d+1})}, \quad \forall u\in C^{\infty}_c(\fR^{d+1})
\end{equation}
under the condition that  there exist constants  $\nu, \gamma>0$ so that for the symbol $\psi(t,\xi)$ of $A(t)$ (i.e. $\cF (A(t)u)(\xi)=\psi(t,\xi)\cF(u)(\xi)$)
it holds that
\begin{equation}
           \label{main con1}
    \Re [\psi(t,\xi)]  \leq - \nu |\xi|^\gamma, \quad \forall \, \xi\in \fR^d\setminus\{0\}
\end{equation}
and for any multi-index $|\alpha| \leq \lfloor \frac{d}{2} \rfloor + 1$
\begin{equation}
             \label{main con2}
 |D^\alpha \psi(t,\xi)| \leq \nu^{-1} |\xi|^{\gamma -|\alpha|}, \quad \forall\, \xi\in \fR^d\setminus\{0\}.
\end{equation}
 Based on the  Marcinkiewicz's interpolation theorem and (\ref{real mainresult})  we prove  a generalization of (\ref{int 1}), that is
\begin{equation}
           \label{eqn 7.19.2}
           \|u_t\|_{L_p(\bR^{d+1})}+\|(-\Delta)^{\gamma/2}u\|_{L_p(\bR^{d+1})}\leq N\|u_t-A(t)u\|_{L_p(\bR^{d+1})}, \quad p>1.
           \end{equation}
Using (\ref{eqn 7.19.2}) one can obtain the unique solvability  of the Cauchy problem
$$
u_t=A(t)u+f, \quad t>0\,; \quad u(0,\cdot)=0
$$
in an appropriate $L_p$-space.

  Here are some examples of operators $A(t)$ satisfying  conditions (\ref{main con1}) and (\ref{main con2}). If $A(t)=(-1)^{m-1}\sum_{|\alpha|=|\beta|=m} a^{\alpha \beta}(t) D^{\alpha+\beta}$  is a $2m$-order differential operator then the symbol
$\psi(t,\xi)=(-1)^m \sum_{|\alpha|=|\beta|=m} a^{\alpha \beta}(t)\xi^{\alpha}\xi^{\beta}$ satisfies (\ref{main con1}) and (\ref{main con2}) if $a^{\alpha \beta}(t)$ are bounded complex-valued measurable functions satisfying
$$
\nu |\xi|^{2m} \leq  \sum_{|\alpha|=|\beta|=m}  \xi^\alpha \xi^\beta \Re\left[a^{\alpha \beta}(t)\right].
$$
Our results   cover the operators of the type
$$
A(t)u=\int_{\fR^d} \Big(u(t,x+y)-  u(t,x)-\chi(y) (u (t,x),y) \Big) m(t,y)
\frac{dy}{|y|^{d+\gamma}}
$$
where $\chi(y)=I_{\gamma>1}+I_{\gamma=1}I_{|y|\leq 1}$ and $m(t,y)$ is a nonnegative measurable function satisfying appropriate conditions. See Section \ref{appl} for details and further examples. The issue regarding the compositions and powers of operators is also discussed in Section \ref{appl}. In particular, for any operators $A_1(t)$ and $A_2(t)$  satisfying the prescribed conditions and   constants $a,b>0$, the operator $C(t)=-(-A_1)^a(-A_2)^b$ satisfies the conditions if for instance the symbols of $A_i(t)$ are real-valued.

Actually in this article we prove a   generalized version of (\ref{real mainresult}). We introduce an optimal condition on the kernel $K(t,s,x)$ (see Assumptions \ref{as 1} and \ref{as 2})  so that the inequality
\begin{equation}
           \label{7.2}
\left\|\int^t_{-\infty}\int_{\fR^d}K(t,s,x-y)f(s,y)dyds\right\|_{BMO(\fR^{d+1})}\leq N\|f\|_{L_{\infty}(\fR^{d+1})}
\end{equation}
holds for any  $f\in C^{\infty}_c(\fR^{d+1})$ with constant $N$ independent of $f$. It turns out that if $A(t)$ is an operator with the symbol $\psi(t,\xi)$ satisfying (\ref{main con1}) and (\ref{main con2}) then the kernel $K(t,s,x)$ related to the formula
$$
(-\Delta)^{\gamma/2}u=\int^t_{-\infty} \int_{\fR^d} K(t,s,x-y)f(s,y)dyds, \quad f:=u_t-A(t)u
$$
satisfies our restrictions on the kernel, that is Assumptions \ref{as 1} and \ref{as 2}.

Below is a short description on related works. In the setting of elliptic equations, the BMO-$L^{\infty}$ estimate
\begin{align}
                \label{int 2}
\|K*f\|_{BMO(\fR^d)} \leq N \|f\|_{L_\infty(\fR^d)}
\end{align}
has been well studied
with Calder\'on-Zygmund kernel $K$. See, for instance,  \cite {Gr}. It seems that the tools used in the literature to prove (\ref{int 2}) are not efficient for  parabolic equations.
Beyond BMO-$L^\infty$  estimate, when it comes to elliptic equations, BMO-BMO type estimates have been obtained in quite general setting
(see, for instance, \cite{BB}, \cite{DM}, and \cite{DKS}).
However, to the best of our knowledge,  there is no BMO-$L^{\infty}$ or BMO-BMO  type estimate  for parabolic equations. We only mention that the sharp function estimate of the type
$$
(A(t)u)^{\sharp}(t,x)\leq \varepsilon [\bM (A(t)u)^2]^{1/2} (t,x) + N(\varepsilon) [\bM (u_t-A(t)u)^2]^{1/2}(t,x), \quad \varepsilon>0
$$
for  parabolic equations is introduced e.g. in   \cite{KrPbook} (second order) and \cite{DKh}  ($2m$-order, $m\in \bN$). Here $h^{\sharp}$ and $\bM h$ represent the sharp function and maximal function of $h$ respectively.

To prove (\ref{7.2}), in place  of the duality property of Hardy space $H^1$ typically used in the literature to prove (\ref{int 2}), we employ  only direct computations on the basis of properties of kernels.

Finally we introduce some notation used in the article. As usual $\fR^{d}$ stands for the Euclidean space of points
$x=(x^{1},...,x^{d})$,  $B_r(x) := \{ y\in \fR^d : |x-y| < r\}$  and
$B_r
 :=B_r(0)$.
 For  multi-indices $\alpha=(\alpha_{1},...,\alpha_{d})$,
$\alpha_{i}\in\{0,1,2,...\}$, $x \in \fR^d$, and  functions $u(x)$ we set
$$
 u_{x^{i}}=\frac{\partial u}{\partial x^{i}}=D_{i}u,\quad \quad
D^{\alpha}u=D_{1}^{\alpha_{1}}\cdot...\cdot D^{\alpha_{d}}_{d}u,
$$
$$
x^\alpha = (x^1)^{\alpha_1} (x^2)^{\alpha_2} \cdots (x^d)^{\alpha_d},\quad \quad
|\alpha|=\alpha_{1}+\cdots+\alpha_{d}.
$$
We also use $D^m_x$ to denote a partial derivative of order $m$  with respect to $x$.
For an open set $U \subset \fR^d$ and a nonnegative integer $n$, we write $u \in C^n(U)$  if $u$ is $n$ times continuously differentiable in $U$.
By $C^{\infty}_c(U)$  we denote the set of infinitely differentiable  functions with compact support in $U$.
The standard $L_p$-space on $U$ with Lebesgue measure is denoted by $L_p(U)$.
We use  ``$:=$" to denote a definition. $\lfloor a \rfloor$ is the biggest integer which is less than or equal to $a$.
By $\cF$ and $\cF^{-1}$ we denote the d-dimensional Fourier transform and the inverse Fourier transform, respectively. That is,
$\cF(f)(\xi) := \int_{\fR^{d}} e^{-i x \cdot \xi} f(x) dx$ and $\cF^{-1}(f)(x) := \frac{1}{(2\pi)^d}\int_{\fR^{d}} e^{ i\xi \cdot x} f(\xi) d\xi$.
For a Borel
set $X\subset \fR^d$, we use $|X|$ to denote its Lebesgue
measure and by $I_X(x)$ we denote  the indicator of $A$.

\mysection{Main results}

Fix $\gamma >0$ throughout this article.
For a locally integrable function $h$ on  $\fR^{d+1}$, we define the BMO semi-norm of $h$ on $\fR^{d+1}$ as follows :
\begin{align*}
\|h\|_{BMO(\fR^{d+1})} = \sup_{Q }  \frac{1}{|Q|}  \int_{Q} |h(r,z) - h_{Q}|~drdz,
\end{align*}
where $f_{Q} := \frac{1}{|Q|}  \int_{Q} f(r,z)~drdz$ and the sup is taken all $Q$  of the type
$$
Q=Q_c(t_0,x_0):= (t_0-c^\gamma, t_0+c^\gamma) \times B_c(x_0) , \quad c>0,\,(t_0,x_0)\in \fR^{d+1}.
$$

Let  $K$ be a  measurable function defined on $\fR^{d+2}$ so that  $K(t,s,\cdot)$ is integrable for each $s<t$.   Denote
$$
\hat K(t,s, \xi) = \cF \Big(K(t,s,\cdot) \Big) (\xi),
$$
where $\cF$ denotes the Fourier transform on $\fR^d$.

\begin{assumption}      \label{as 1}
There exists a measurable function $H$ on $\fR^{d+1}$ such that
 for all $t>s$ and $\xi \in \fR^d$,
\begin{align}           \label{as 11}
|\hat K (t,s,\xi)| \leq    H (t-s,\xi)
\end{align}
and
\begin{align}           \label{as 12}
\sup_{\xi} \int_{0}^\infty  H(t,\xi) dt <\infty.
\end{align}

\end{assumption}

\begin{assumption}      \label{as 2}
There exists a nondecreasing function $\varphi(t):(0,\infty)\to [0,\infty)$ such that

(i) for any $s>r$  and $c>0$,

\begin{align}           \label{as 21}
\int_{r}^s \int_{|z| \geq c} |K(s,\tau, z)| ~dz d\tau \leq \varphi((s-r)c^{-\gamma})\,;
\end{align}

  (ii) for any $s>r>a$,
\begin{align}                   \label{as 32}
\int_{-\infty}^{a} \int_{\fR^d} | K(s,\tau, z) - K(r,\tau, z)| ~dz d\tau \leq \varphi((s-r)(r-a)^{-1});
\end{align}

  (iii) for any $s > a $ and $ h \in \fR^d$,
\begin{align}               \label{as 31}
\int_{-\infty}^{a} \int_{\fR^d} | K(s,\tau, z+h)-K(s,\tau, z)| ~dz d\tau \leq \varphi\big( |h|(s-a)^{-1/\gamma} \big).
\end{align}
\end{assumption}

Note that
\begin{align*}
&\int_{-\infty}^{a} \int_{\fR^d} | K(s,\tau, z+h)-K(s,\tau, z)| ~dz d\tau \\
&=\int_{s-a}^{\infty} \int_{\fR^d} | K(s,s-\tau, z+h)-K(s,\tau, z)| ~dz d\tau.
\end{align*}

Thus, if $K(s,\tau,z)=K(s-\tau,z)$ then (\ref{as 31}) is equivalent to
\begin{align*}
\int_{b}^{\infty} \int_{\fR^d} | K(\tau, z+h)-K(\tau, z)| ~dz d\tau \leq \varphi\big( |h|b^{-1/\gamma} \big).
\end{align*}

For a function $f$ on $\fR^{d+1}$, denote
\begin{align}
                \label{g definition}
\cG  f(t,x) :=  \int_{-\infty}^t K(t,s, \cdot) \ast f(s, \cdot) (x)~ds.
\end{align}
\begin{remark}
                    \label{well define g}
If $f$ has compact support and is regular enough with respect $x$, then $\cG f$ is well defined.
For instance, one can check that if $f\in C^{\infty}_c(\fR^{d+1})$ then for any multi-index $\alpha$,
$$
\sup_{s,\xi}|\xi^{\alpha}\hat{f}(s,\xi)|=\sup_{s,\xi}|\cF\big(D^{\alpha}f(s,\cdot)\big)(\xi)|<\infty.
$$
 Therefore $\sup_s |\hat{f}(s,\xi)|\in L_1(\fR^d)$ and from \eqref{as 12},
\begin{align*}
\int^t_{-\infty} |K(t,s,\cdot)\ast f(s,\cdot)(x)|ds
&= \int^t_{-\infty}|\cF^{-1}(\hat{K}(t,s,\xi)\hat{f}(s,\xi))(x)|ds\\
&\leq \int^t_{-\infty}\int_{\fR^d}H(t-s,\xi)|\hat{f}(s,\xi)|\,d\xi ds\\
&= \int_{\fR^d}|\sup_s \hat{f}(s,\xi)| \left(\sup_{\xi}\int^{\infty}_{0} H(t,\xi)dt\right)\,d\xi<\infty.
\end{align*}
It follows that $\cG f$ is well defined for  functions $f\in C^{\infty}_c(\fR^{d+1})$.
\end{remark}

Theorems \ref{BMO theorem} and \ref{pseudo thm} below are our main results. The proofs of the theorems are given in Sections \ref{proof BMO theorem}  and  \ref{proof pseudo thm}.

\begin{thm}
                    \label{BMO theorem}
 Let  Assumptions \ref{as 1} and  \ref{as 2} hold and $p \in [2,\infty)$.
Then for any $f\in C^{\infty}_c(\fR^{d+1})$ it holds that

\begin{equation}
                     \label{bmo part1}
\|\cG f \|_{BMO(\fR^{d+1})} \leq N \|f\|_{L_\infty(\fR^{d+1})}
\end{equation}
and
\begin{equation}
                 \label{bmo part2}
\|\cG f \|_{L_p( \fR^{d+1} )} \leq N \|f\|_{L_p( \fR^{d+1} )},
\end{equation}
where the constant  $N$ depends only on $d$, $p$, and the constants in the assumptions.
\end{thm}

\vspace{4mm}

Next, we formulate the conditions on the pseudo-differential operators $A(t)$
such that the kernels $K(t,s,x)$ related to $A(t)$ satisfy Assumptions \ref{as 1} and \ref{as 2}. Let $A(t)$ be an  operator with the symbol $\psi(t,\xi)$, that is
$$
\cF (A(t)u)(\xi)=\psi(t,\xi) \cF(u)(\xi), \quad \forall \, u\in C^{\infty}_c (\fR^d).
$$
Define the kernel $p(t,s,x)$ by the formula
\begin{align*}
p(t,s,x)=I_{s<t } \cF^{-1} \Big( \exp \big(\int_s^t \psi(r,\xi))dr \big) \Big)(x),
\end{align*}
so that  the solution of the equation
\begin{equation*}
\frac{\partial u}{\partial t}=A(t)u+f, \quad t\in \fR
\end{equation*}
is (formally) given by
\begin{equation*}
u(t)=\int^t_{-\infty} (p(t,s,\cdot)*f(s,\cdot))(x) ds.
\end{equation*}

Denote
$$
K(t,s,x)=(-\Delta)^{\gamma/2}p(t,s,x)
$$
and
$$
\cG  f(t,x) := (-\Delta)^{\gamma/2}u:=  \int_{-\infty}^t K(t,s, \cdot) \ast f(s, \cdot) (x)~ds.
$$

By $\Re z$ we denote the real part of $z$.

\begin{assumption}
                     \label{as 7.8}
There exists a constant $\nu >0$
such that for any $t \in \fR, \xi\in \fR^d\setminus \{0\}$ and  multi-index $|\alpha| \leq  \lfloor \frac{d}{2} \rfloor+1$,
\begin{equation}
     \label{A1}
      \Re [\psi(t,\xi)]  \leq - \nu |\xi|^\gamma,\quad \quad
 |D^\alpha \psi(t,\xi)| \leq \nu^{-1} |\xi|^{\gamma -|\alpha|}.
 \end{equation}
\end{assumption}

\begin{thm}         \label{pseudo thm}
  Let Assumption \ref{as 7.8} hold and $p>1$. Then

(i)  for any $ f \in C^{\infty}_c(\fR^{d+1})$,
$$
\|\cG f\|_{BMO(\fR^{d+1})}\leq N(\nu,\gamma,d) \|f\|_{L_{\infty}(\fR^{d+1})}\, ;
$$

(ii) for any $u\in C^{\infty}_c(\fR^{d+1})$,
\begin{equation}
                         \label{7.9.1}
\|u_t\|_{L_p(\fR^{d+1})}+\| (-\Delta)^{\gamma/2}u\|_{L_p(\fR^{d+1})} \leq N(p,\nu,\gamma,d) \|u_t-A(t)u\|_{L_p(\fR^{d+1})}.
\end{equation}
\end{thm}

\mysection{Some fundamental estimates}

In this section we estimate the mean oscillation of $\cG f$ in terms of $\|f\|_{L_{\infty}}$.
Recall that
$$
\cG  f(t,x) :=  \int_{-\infty}^t K(t,s, \cdot) \ast f(s, \cdot) (x)~ds.
$$

We first derive an $L_2$ estimate of $\cG f$.
\begin{lemma}       \label{l2 lemma}
Suppose that Assumption \ref{as 1} holds and $f\in C^{\infty}_c(\fR^{d+1})$.
Then
$$
\|\cG f \|_{L_2( \fR^{d+1})} \leq N \| f\|_{L_2( \fR^{d+1})},
$$
where the constant $N$ is independent of $f$. Consequently, the map $f\to \cG f$ is extendable to a bounded linear operator  on $L_2(\fR^{d+1})$.
\end{lemma}

\begin{proof}
By Parseval's identity,
\begin{align*}
&\int_{-\infty}^\infty \int_{\fR^d} |\cG f(t,x)|^2 dx dt\\
&=N\int_{-\infty}^\infty \int_{\fR^d} \Big|\int_{-\infty}^t \hat K(t,s,\xi) \hat f (s,\xi)~ds \Big|^2  d\xi dt \\
& \leq N\int_{-\infty}^\infty \int_{\fR^d} \Big|\int_{-\infty}^\infty  I_{ s < t} |\hat K(t,s,\xi)| |\hat f (s,\xi)|~ds \Big|^2  d\xi dt.
\end{align*}
Hence it follows from Assumption \ref{as 1} and Parseval's identity that
\begin{align*}
&\int_{-\infty}^\infty \int_{\fR^d} |\cG f(t,x)|^2 dxdt\\
&\leq N\int_{-\infty}^\infty \int_{\fR^d} \Big|\int_{-\infty}^\infty  I_{s < t} H(t-s,\xi) |\hat f (s,\xi)|~ds \Big|^2  d\xi dt \\
&=N\int_{-\infty}^\infty \int_{\fR^d} \Big|\int_{-\infty}^\infty e^{i t \tau} \int_{-\infty}^\infty   I_{ s < t} H(t-s,\xi)  |\hat f (s,\xi)|~dsdt \Big|^2    d \xi  d\tau\\
&=N\int_{-\infty}^\infty \int_{\fR^d} \Big|\int_{-\infty}^\infty \int_{-\infty}^\infty e^{i t \tau}  I_{ s < t} H(t-s,\xi) dt |\hat f (s,\xi)|~ds \Big|^2    d \xi  d \tau\\
&\leq N\int_{-\infty}^\infty \int_{\fR^d} \Big|\int_{0}^\infty e^{i t \tau}  H(t,\xi) dt\Big|^2 \Big|\int_{\fR}  e^{i s \tau} |\hat f (s,\xi)|~ds \Big|^2    d \xi  d \tau\\
&\leq N\int_{-\infty}^\infty \int_{\fR^d}  \Big|\int_{\fR}  e^{i s \tau} |\hat f (s,\xi)|~ds \Big|^2    d \xi d\tau\\
&= N \int_{-\infty}^\infty \int_{\fR^d}    |\hat f (s,\xi)|^2   d \xi ds= N\int_{-\infty}^\infty \int_{\fR^d} |f (s,x)|^2   dx ds.
\end{align*}
The lemma is proved.
\end{proof}

For the rest of this section, $\cG$ is understood as a bounded linear operator on $L_2(\fR^{d+1})$.

\begin{corollary}       \label{inside sup est}
Let   $f\in  L_2(\fR^{d+1})$ and vanish on $\fR^{d+1} \setminus Q_{3c}(t_0,0)$.
Suppose  that Assumption \ref{as 1} holds.
Then
\begin{align*}
\int_{Q_c(t_0,0)} |  \cG f (s,y)|~ds dy \leq   N|Q_{c}| \cdot \sup_{Q_{3c}(t_0,0)} | f|,
\end{align*}
where $N$ does not depend on $c, t_0$ and $f$.
\end{corollary}

\begin{proof}
By H\"older's inequality and Lemma \ref{l2 lemma},
\begin{align*}
\int_{Q_c(t_0,0)} |\cG f (s,y)|~ds dy
&\leq \Big(\int_{Q_c(t_0,0)} |\cG f (s,y)|^2~ds dy \Big)^{1/2} |Q_c|^{1/2} \\
&\leq \Big(\int_{\fR^{d+1}} |\cG f (s,y)|^2~ds dy \Big)^{1/2} |Q_c|^{1/2} \\
&\leq \Big(\int_{\fR^{d+1}} |f (s,y)|^2~ds dy \Big)^{1/2} |Q_c|^{1/2} \\
&= \Big(\int_{  Q_{3c}(t_0,0)} |f (s,y)|^2~ds dy \Big)^{1/2} |Q_c|^{1/2} \\
&\leq N |Q_{c}|  \sup_{Q_{3c}(t_0,0)} |f|.
\end{align*}
The lemma is proved.
\end{proof}

In the following lemma we estimate the mean oscillation of $\cG f$ on $Q_c(t_0,0)$ when $f$ vanishes near $Q_c(t_0,0)$.

\begin{lemma}       \label{inside lemma}
Suppose that Assumption \ref{as 2} holds.
Let $f\in  L_2(\fR^{d+1})$ and $f=0$ on $Q_{2c}(t_0,0)$. Then
\begin{align}
                    \label{mean osc in zero}
\int_{Q_c(t_0,0)} \int_{Q_c(t_0,0)} |  \cG f (s,y)- \cG f (r,z)|~dsdr dydz  \leq   N|Q_{c}|^2 \cdot  \sup_{\fR^{d+1}} | f|,
\end{align}
where $N$ does not depend on $c, t_0$ and $f$.
\end{lemma}

\begin{proof}
First we assume $f \in C_c^\infty (\fR^{d+1})$. We will prove
\begin{align}
                    \label{mean osc in zero 0}
\int_{Q_c(t_0,0)} |  \cG f (s,y)- \cG f (t_0-c^\gamma,0)|~ds dy  \leq   N|Q_{c}| \cdot  \sup_{\fR^{d+1}} | f|.
\end{align}
Let $(s,y) \in Q_c(t_0,0)$. Then
\begin{align*}
&|\cG f (s,y)- \cG f (t_0-c^\gamma,0)|\\
&\leq | \cG f (s,y)-  \cG f(s,0)| + | \cG f(s,0)- \cG f (t_0-c^\gamma,0)|\\
&=: \cI_1 + \cI_2
\end{align*}

We consider $\cI_1$ first.
\begin{eqnarray*}
\cI_1
&=& \Big| \int_{-\infty}^{s} \int_{\fR^d}  \big(K(s,\tau, y-z)-K(s,\tau,-z)\big) f(\tau, z)~dz d\tau\Big| \\
&=& \Big|\int^{s}_{t_0-(2c)^{\gamma}}\int_{\fR^d} \cdots \,dzd\tau +\int_{-\infty}^{t_0-(2c)^{\gamma}} \int_{\fR^d}\cdots \,dzd\tau \Big|\\
&\leq&  \int_{t_0-(2c)^{\gamma}}^{s} \int_{\fR^d}  |K(s,\tau, z)| |f(\tau, y-z)|~dz d\tau \\
&&+ \int_{t_0-(2c)^{\gamma}}^{s} \int_{\fR^d}  |K(s,\tau, z)| |f(\tau ,-z)|~dz  \\
&&+\int_{-\infty}^{t_0-(2c)^{\gamma}} \int_{\fR^d}  \big|K(s,\tau, y-z)-K(s,\tau,-z)\big| |f(\tau, z)|~dz d\tau \\
&=:&\cI_{11}+\cI_{12}+\cI_{13}.
\end{eqnarray*}
Note that if $t_0 -(2c)^{\gamma} < \tau  \leq s \leq  t_0+ c^{\gamma}$ and $|z| \leq c$, then
\begin{align}           \label{zero con}
f(\tau,  y -z) =0~\text{and}~f(\tau,-z)=0,
\end{align}
 because $|y-z| \leq 2c$ and $|-z| \leq c$, and $f=0$ on $Q_{2c}(t_0,0)$.
Hence by \eqref{as 21}, $\cI_{11}+\cI_{12}$ is less than or equal to
\begin{align*}
&N \sup_{\fR^{d+1}}|f| \int_{t_0-(2c)^{\gamma}}^{s} \int_{|z| \geq c} | K(s,\tau, z)| ~dz d\tau\\
&\leq N \varphi([s- (t_0-(2c)^\gamma)]c^{-\gamma}) \sup_{\fR^{d+1}}|f| \leq N \sup_{\fR^{d+1}}|f| .
\end{align*}
Also, by
 \eqref{as 31},
\begin{align*}
\cI_{13}&\leq N \sup_{\fR^{d+1}}\, |f|  \int_{-\infty}^{t_0-(2c)^{\gamma}} \int_{\fR^d} \big|K(s,\tau, y-z)-K(s,\tau,-z)\big| ~dz d\tau\\
&\leq N \varphi\big(c (s-t_0+(2c)^{\gamma})^{-1/\gamma} \big) \sup_{\fR^{d+1}}|f|\\
&\leq N \sup_{\fR^{d+1}}|f|.
\end{align*}

Next, we consider $\cI_2$. Note that
\begin{align*}
\cI_2&=\Big|  \cG f (s,0) -  \cG f(t_0-c^\gamma,0) \Big| \\
&= \Big| \int_{-\infty}^{s} \int_{\fR^d}  K(s,\tau, z)  f(\tau,-z)~dz d\tau - \int_{-\infty}^{t_0-c^\gamma} \int_{\fR^d}  K(t_0-c^\gamma , \tau, z)  f(\tau,-z)~dz d\tau \Big|\\
&\leq \Big| \int_{-\infty}^{s} \int_{\fR^d} K(s,\tau, z)  f(\tau,-z)~dz d\tau - \int_{-\infty}^{t_0-c^\gamma} \int_{\fR^d} K(s, \tau, z)  f(\tau,-z)~dz d\tau \Big|\\
&+\Big| \int_{-\infty}^{t_0-c^\gamma} \int_{\fR^d} \left[K(s,\tau, z)-K(t_0-c^\gamma, \tau, z) \right] f(\tau,-z)~dz d\tau \Big|\\
&=: \cI_{21} + \cI_{22}.
\end{align*}
Recall that $f=0$ on $[t_0-(2c)^\gamma,t_0+(2c)^\gamma] \times B_{2c}$.
So by \eqref{as 21}
\begin{align*}
\cI_{21}
&\leq  \int_{t_0-c^\gamma}^{s} \int_{\fR^d} |K(s,\tau, z)| | f|(\tau,-z)~dz d\tau \\
&\leq  \sup_{\fR^{d+1}}| f| \int_{t_0-c^\gamma}^{s} \int_{|z| \geq c} | K(s,\tau, z)| ~dz d\tau \\
&\leq  N\varphi([s -(t_0-c^\gamma)]c^{-\gamma})\,\sup_{\fR^{d+1}}| f|  \leq N\sup_{\fR^{d+1}}| f|.
\end{align*}
Also,
\begin{align*}
\cI_{22}
&\leq  \int_{t_0-(2c)^{\gamma}}^{t_0-c^\gamma} \int_{\fR^d} | K(s,\tau, z) - K(t_0-c^\gamma, \tau, z)| |f(\tau,-z)|dz d\tau \\
&+  \sup_{\fR^{d+1}} | f| \int_{-\infty}^{t_0-(2c)^{\gamma}} \int_{\fR^d} \Big| K(s,\tau, z) - K(t_0-c^\gamma, \tau, z) \Big| dz d\tau \\
&=: \cI_{221}+ \cI_{222}.
\end{align*}
Recalling \eqref{zero con}, by \eqref{as 21} we have
\begin{eqnarray*}
\cI_{221}
&\leq&  \sup_{\fR^{d+1}} | f| \int_{t_0-(2c)^{\gamma}}^{s} \int_{|z| \geq c} | K(s,\tau, z)|~dz d\tau \\
&&+ \sup_{\fR^{d+1}}|f|\int_{t_0-(2c)^{\gamma}}^{t_0-c^\gamma} \int_{|z| \geq c} | K(t_0-c^\gamma, \tau, z) | dz d\tau\\
&\leq& N \sup_{\fR^{d+1}}|f|.
\end{eqnarray*}
On the other hand, by \eqref{as 32}, we obtain
$$
\cI_{222} \leq \varphi\big([s -(t_0-c^\gamma)] (2^\gamma -1)^{-1} c^{-\gamma}\big)\sup_{\fR^{d+1}}|f| \leq N \sup_{\fR^{d+1}}|f|.
$$
Hence \eqref{mean osc in zero 0} is proved and this obviously implies \eqref{mean osc in zero} for $f \in C_c^\infty(\fR^{d+1})$.

 Now we consider the general case, that is  $f \in L_2(\fR^{d+1})$.  For given $\varepsilon>0$  we choose a sequence of functions $f_n \in C_c^\infty(\fR^{d+1})$ such that $f_n=0$ on $Q_{(2-2\varepsilon)c}(t_0,0)$,
$\cG f_n \to \cG f~(a.e.)$ and $ \sup_{\fR^{d+1}} |f_n| \leq \sup_{\fR^{d+1}} |f|$.
Then by Fatou's theorem,
\begin{align*}
&\int_{Q_{(1-\varepsilon)c}(t_0,0)}\int_{Q_{(1-\varepsilon)c}(t_0,0)} |  \cG f (s,y)- \cG f (r,z)|~dsdr dydz \\
&\leq \liminf_{n \to \infty} \int_{Q_{(1-\varepsilon)c}(t_0,0)} \int_{Q_{(1-\varepsilon)c}(t_0,0)} |  \cG f_n (s,y)- \cG f_n (r,z)|~dsdr dydz \\
&\leq   N|Q_{c}|^2 \cdot  \liminf_{n \to \infty} \sup_{\fR^{d+1}} | f_n| \leq N|Q_{c}|^2 \cdot  \sup_{\fR^{d+1}} | f|.
\end{align*}
Since $\varepsilon$ is arbitrary the lemma is proved.
\end{proof}

We introduce a simple decomposition of $f$. For any $\lambda >0$ set
$$
f_{1,\lambda}(t,x) := f(t,x) I_{|f| > \lambda} , \quad f_{2,\lambda}(t,x) :=f(t,x)I_{|f| \leq \lambda}.
$$

The following lemma is  a modified version of  Marcinkiewicz's interpolation theorem. We provide a proof for the sake of completeness.

\begin{lemma}
                \label{ma interpol}
Let $\cA$ be a subadditive operator on $L_2(\fR^{d+1}) \cap L_\infty(\fR^{d+1})$ and
 $f \in L_2(\fR^{d+1}) \cap L_\infty(\fR^{d+1})$.
Suppose that
\begin{align}
                \label{lambda as 1}
\|\cA (f_{1,\lambda}) \|_{L_2(\fR^{d+1})} \leq N_1\| f_{1,\lambda} \|_{L_2(\fR^{d+1})}
\end{align}
and
\begin{align}
                    \label{lambda as 2}
\|\cA (f_{2,\lambda}) \|_{L_\infty(\fR^{d+1})} \leq N_2\| f_{2,\lambda} \|_{L_\infty(\fR^{d+1})}
\end{align}
for all $\lambda >0$. Then for $p \in (2, \infty)$ we have
\begin{align*}
\|\cA f \|_{L_p(\fR^{d+1})} \leq N\| f \|_{L_p(\fR^{d+1})},
\end{align*}
where $N$ depends only on $d$, $p$, $N_1$, and $N_2$.
\end{lemma}
\begin{proof}
Note that by Fubini's theorem
\begin{align}
                \label{dist decom}
\|\cA f\|^p_{L_p(\fR^{d+1})}
=N\int_0^\infty \big|\{(t,x) : |\cA f(t,x)| > 2N_2 \lambda \}\big| \lambda^{p-1}~d\lambda.
\end{align}
Since for each $\lambda >0$, $f = f_{1,\lambda} + f_{2,\lambda}$ and $\cA$ is subadditive,
\begin{align*}
&\big|\{(t,x) : |\cA f(t,x)| > 2N_2\lambda \}\big|  \\
&\leq \big|\{(t,x) : |\cA f_{1,\lambda}(t,x)| > N_2\lambda \}\big| + \big|\{(t,x) : |\cA f_{2,\lambda}(t,x)| > N_2\lambda \}\big|.
\end{align*}
Due to \eqref{lambda as 2},
$$
\|\cA (f_{2,\lambda}) \|_{L_\infty(\fR^{d+1})} \leq N_2\| f_{2,\lambda} \|_{L_\infty(\fR^{d+1})} \leq N_2\lambda,
$$
which clearly implies
$$
\big|\{(t,x) : |\cA f_{2,\lambda}(t,x)| > N_2\lambda \}\big| =0.
$$
Moreover by \eqref{lambda as 1} and Chebyshev's inequality,
$$
\big|\{(t,x) : |\cA f_{1,\lambda}(t,x)| > N_2\lambda \}\big| \leq N\frac{1}{\lambda^2} \|f_{1,\lambda}\|^2_{L_2(\fR^{d+1})}.
$$
Hence going back to \eqref{dist decom}, we get

\begin{align*}
\|\cA f\|^p_{L_p(\fR^{d+1})}
&\leq N\int_0^\infty \frac{1}{\lambda^2} \|f_{1,\lambda}\|^2_{L_2(\fR^{d+1})} \lambda^{p-1}~d\lambda \\
&\leq N \int_{\fR^{d+1}} |f(t,x)|^2  \int_0^\infty I_{|f| > \lambda}\lambda^{p-3}~ d\lambda dtdx  \\
&\leq N \int_{\fR^{d+1}} |f(t,x)|^p dtdx.
\end{align*}
The lemma is proved.
\end{proof}

\mysection{Proof of Theorem \ref{BMO theorem}}              \label{proof BMO theorem}

{\bf{Part I}}. We first prove (\ref{bmo part1}) for   $f \in L_2(\fR^{d+1}) \cap L_\infty(\fR^{d+1})$.
It suffices to
prove that for each $Q=Q_c(t_0,x_0)$
\begin{align*}
\dashint_Q |\cG f- (\cG f)_Q|~dsdy \leq N \sup_{\fR^{d+1}}|f|.
\end{align*}
Moreover, since $\cG f(\cdot,\cdot)(t,x+x_0)=\cG f(\cdot,x_0+\cdot)(t,x)$, considering a translation we may assume that $x_0=0$. Thus
$$
Q=Q_c(t_0,x_0)=(t_0-c^{\gamma}, t_0+c^{\gamma})\times B_c(0).
$$

Take $\zeta \in C_c^\infty(\fR^{d+1})$ such that $\zeta =1 $ on $Q_{2c}$ and $\zeta=0$ outside of $Q_{3c}$.
Then
\begin{align*}
&\dashint_Q |\cG f- (\cG f)_Q|~dsdy \\
&\leq  2\dashint_Q |\cG (f \zeta)|~dsdy +  \dashint_Q \dashint_Q |\cG (f (1-\zeta))(s,y)-\cG (f (1-\zeta))(r,z)|~dsdrdydz \\
&=: \cI_1 + \cI_2.
\end{align*}
Due to Corollary \ref{inside sup est},
\begin{align*}
\cI_1 \leq
N\frac{1}{|Q|}\int_Q |\cG (f \zeta)(s,y)|~dsdy
\leq N\sup_{\fR^{d+1}} |f \zeta| \leq N\sup_{\fR^{d+1}} |f|.
\end{align*}
On the other hand, by Lemma \ref{inside lemma} we have
$$
\cI_2 \leq N \sup_{\fR^{d+1}} |f(1-\zeta)| \leq N \sup_{\fR^{d+1}} |f|.
$$
Hence for any $f \in L_2(\fR^{d+1}) \cap L_\infty(\fR^{d+1})$ we have
\begin{equation}
                    \label{general bmo}
\|\cG f \|_{BMO(\fR^{d+1})} \leq N \|f\|_{L_\infty(\fR^{d+1})},
\end{equation}
where $N$ is independent of $f$.
Therefore \eqref{bmo part1} is proved.

\vspace{3mm}

{\bf{Part II}}. Next we prove (\ref{bmo part2}). For a measurable function $h(t,x)$ on $\fR^{d+1}$, we define the maximal function
$$
\cM h(t,x)
= \sup_{ Q }  \frac{1}{| Q|}  \int_{ Q} |f(r,z)|~drdz,
$$
and the sharp function $h^\sharp(t,x)$
\begin{align*}
h^\sharp(t,x)
= \sup_{ Q }  \frac{1}{| Q|}  \int_{ Q} |f(r,z) - f_{ Q}|~drdz,
\end{align*}
where $f_{ Q} := \frac{1}{| Q|}  \int_{ Q} f(r,z)~drdz$, and the sup is taken all $ Q$ containing $(t,x)$ of the type
$$
 Q= Q_c(t_0,x_0):=(t_0-c^\gamma,t_0+c^\gamma) \times B_c(x_0), \quad c>0,~(t_0,x_0) \in \fR^{d+1}.
$$
Then by Fefferman-Stein theorem \cite[Theorem 4.2.2]{Ste}, for any $h\in L_p(\fR^{d+1})$,
\begin{align*}
\|h\|_{L_p(\fR^{d+1})} \leq N \|h^{\sharp}\|_{L_p(\fR^{d+1})}.
\end{align*}
Moreover, by  Hardy-Littlewood maximal theorem and the inequality $
|h^\sharp (t,x)| \leq 2 \cM h(t,x)$,
\begin{align}   \label{hardy littlewood}
\|h^\sharp\|_{L_p(\fR^{d+1})} \leq N \|\cM h\|_{L_p(\fR^{d+1})} \leq N\|h\|_{L_p(\fR^{d+1})}.
\end{align}
Combining  Lemma \ref{l2 lemma} with \eqref{hardy littlewood}, we get for  any $f \in L_2(\fR^{d+1})$,
$$
\|( \cG f)^\sharp\|_{L_2(\fR^{d+1})} \leq N \|f\|_{L_2(\fR^{d+1})}.
$$
Moreover  by \eqref{general bmo},
\begin{align}
\|(\cG f)^\sharp\|_{L_{\infty}(\fR^{d+1})}\leq N \|f\|_{L_{\infty}(\fR^{d+1})}.
\end{align}
Note that the  map  $f \to (\cG f)^\sharp$ is subadditive since $\cG$ is a linear operator.
Hence by Lemma \ref{ma interpol} for any $p \in [2, \infty)$ there exists a constant $N$ such that
$$
\|( \cG f)^\sharp\|_{L_p(\fR^{d+1})} \leq N \|f\|_{L_p(\fR^{d+1})}, \quad \forall\, f\in L_2(\fR^{d+1}) \cap  L_\infty(\fR^{d+1}).
$$
Finally by Fefferman-Stein theorem, we get
$$
\|\cG f\|_{L_p(\fR^{d+1})} \leq N \|f\|_{L_p(\fR^{d+1})},
$$
where $N$ is independent of $f$. Therefore (\ref{bmo part2}) is proved.

\mysection{Proof of Theorem \ref{pseudo thm}}       \label{proof pseudo thm}

Recall that $A(t)$ is a pseudo differential operator with the symbol $\psi(t,\xi)$ satisfying
$$
      \Re [\psi(t,\xi)]  \leq - \nu |\xi|^\gamma,\quad
 |D^\alpha \psi(t,\xi)| \leq \nu^{-1} |\xi|^{\gamma -|\alpha|}
$$
for any multi-index $|\alpha| \leq  \lfloor \frac{d}{2} \rfloor+1$. Also recall $p(t,s,x)$ and $K(t,s,x)$ are defined by
$$
p(t,s,x)=I_{s<t } \cF^{-1} \Big( \exp \big(\int_s^t \psi(r,\xi)dr \big) \Big)(x), \quad K(t,s,x)=(-\Delta)^{\gamma/2}p(t,s,x).
$$

In this section we prove that $K(t,s,x)$ satisfies Assumptions \ref{as 1} and \ref{as 2} using the
 following auxiliary results.

\begin{lemma}   \label{frac est}
Let   $h \in C^2(\fR^d \setminus \{0\})$ satisfy
\begin{align}   \label{ck}
|h(x)| \leq N_0|x|^\varsigma e^{-c|x|^\gamma},  \quad \forall x\in\fR^{d} \setminus \{0\},
\end{align}
with some    constants $c, N_0>0$, $\varsigma > \eta-\frac{d}{2}$ and $\gamma > 0$. Further assume that either
\begin{align}           \label{ddck}
 \eta\in [0,1)\quad \text{and}\quad \big|D h(x) \big| \leq N_0|x|^{\varsigma-1} e^{-c|x|^\gamma},  \quad \forall x \in \fR^{d} \setminus \{0\}
\end{align}
 or
 \begin{equation}
              \label{ddck2}
\eta\in [1,2) \quad \text{and}\quad \big|D^2h(x) \big| \leq N_0|x|^{\varsigma-2} e^{-c|x|^\gamma},  \quad \forall x \in \fR^{d} \setminus \{0\}
\end{equation}
 holds.
 Then
$$
\|(-\Delta)^{\eta/2}h\|_{L_2(\fR^d)} < N< \infty,
$$
where  $N=N(N_0,\eta,c,\varsigma,\gamma)$.
\end{lemma}

\begin{proof}
We assume $\eta \in (0,2)$ since the statement is obvious if $\eta=0$. We further assume  $\varsigma<\eta$ because if
(\ref{ck})-(\ref{ddck2}) hold for some $\varsigma$ then they hold for any $\varsigma'\leq \varsigma$ (with other constant $N_0$).

{\bf{Case 1}}. Suppose (\ref{ddck}) holds.
Let   $C=C(\eta)>0$ be the constant  such that
$$
-(-\Delta)^{\eta/2} h(x) = C\lim_{\varepsilon \to 0} \int_{|y| \geq \varepsilon}  \frac{ h(x+y) - h(x)}{|y|^{d+\eta}}~dy=\cI(x)+\cJ(x),
$$
where
$$
\cI(x)= C \int_{|y| \geq |x|/2}  \frac{ h(x+y) - h(x)}{|y|^{d+\eta}}~dy
$$
and
$$
\cJ(x)= C\lim_{\varepsilon \to 0}\int_{|x|/2>|y| \geq \varepsilon}  \frac{ h(x+y) - h(x)}{|y|^{d+\eta}}~dy.
$$
Obviously,
\begin{align*}
|\cI(x)|
\leq  C\int_{|y| \geq |x|/2}  \frac{|h(x+y)|}{|y|^{d+\eta}}~dy +C\int_{|y| \geq |x|/2}  \frac{ |h(x)|}{|y|^{d+\eta}}~dy=:\cI_1(x) + \cI_2(x).
\end{align*}
Recall $\eta > \varsigma$. From \eqref{ck}, if $|x| < 1$
\begin{align}
\cI_{1}
\notag \leq C \int_{|y| \geq |x|/2}  \frac{ |x+y|^\varsigma }{|y|^{d+\eta}}~dy &= C|x|^{-\eta}\int_{|y| \geq 1/2}  \frac{ |x+|x|y|^{\varsigma} }{|y|^{d+\eta}}~dy \\
\notag &\leq C |x|^{\varsigma - \eta} \sup_{|w|=1}\int_{ |y| \geq 1/2}  \frac{ |w+y|^\varsigma }{|y|^{d+\eta}}~dy \\
 &\leq N|x|^{\varsigma - \eta},  \label{i1}
\end{align}
where the last inequality is from the condition $\varsigma>\eta-d/2>-d$. On the other hand, if $|x| \geq 1$ (recall $\varsigma>-d$ and $\gamma>0$)
\begin{align}
\cI_1
\notag &\leq C\int_{|y| \geq |x|/2}  \frac{ |x+y|^{\varsigma} e^{-c|x+y|^\gamma}}{|y|^{d+\eta}}~dy \\
\label{i12} &\leq N\frac{1}{|x|^{d+\eta}} \int_{\fR^d}  |y|^{\varsigma} e^{-c|y|^\gamma}~dy
\leq N\frac{1}{|x|^{d+\eta}}.
\end{align}
Also, using \eqref{ck} again, we get
\begin{align}
\label{i2} \cI_2 \leq  N|x|^{\varsigma-\eta}  e^{-c|x|^\gamma}, \quad \quad  \forall\, x \in \fR^d.
\end{align}

To estimate $\cJ$, we use Taylor's theorem and get
\begin{align*}
|\cJ|
\leq  N\int_{|y|<|x|/2} \left|\nabla h(x+\bar \theta y) \right| \frac{1}{|y|^{d-1+\eta}}~dy,
\end{align*}
where $0 \leq \bar \theta \leq 1$.
So from \eqref{ddck},
\begin{align}
\label{j} |\cJ|
\leq N|x|^{\varsigma - \eta} e^{-c(\frac{1}{2})^{\gamma}|x|^\gamma}, \quad \quad \forall\, x \in \fR^d.
\end{align}
Therefore by \eqref{i1}, \eqref{i12}, \eqref{i2}, and \eqref{j}, we have
\begin{align*}
\int_{\fR^d}|(-\Delta)^{\eta/2} h( x)|^2 ~dx < \infty
\end{align*}
because $ \varsigma > \eta-\frac{d}{2}$ and $\gamma >0$.

{\bf{Case 2}}. Suppose (\ref{ddck2}) holds. The proof for this case  is very close to  Case 1.
It is enough to repeat the above proof, but  in order to estimate $\cJ$  we  use the second order Taylor's theorem in stead of the first order one.
\end{proof}


Before going further, for the simplicity of presentation we define
\begin{align*}
q_1(t,s,x)=I_{ s<t } \cF^{-1} \Big( \exp \big(\int_s^t \psi(r,(t-s)^{-1/\gamma}\xi)dr \big) \Big)(x),
\end{align*}
and
\begin{align*}
&q_2(t,s,x) \\
&=(t-s)I_{ s<t } \cF^{-1} \Big( \psi(t,(t-s)^{-1/\gamma}\xi) |\xi|^\gamma \exp \big(\int_s^t \psi(r,(t-s)^{-1/\gamma}\xi)dr \big) \Big)(x).
\end{align*}
There are following relations among $p$, $q_1$, and $q_2$ :
$$
(t-s)^{d/\gamma}p(t,s,(t-s)^{1/\gamma}x)=q_1(t,s,x),
$$
\begin{align}
                    \label{rela p q1}
(t-s)^{d/\gamma} (t-s)\Delta^{\gamma/2}p(t,s,(t-s)^{1/\gamma}x)=\Delta^{\gamma/2}q_1(t,s,x),
\end{align}
and
\begin{align}
                    \label{rela p q2}
\frac{\partial}{\partial t} \Delta^{\gamma/2} p(t,s,x) =(t-s)^{-d/\gamma}(t-s)^{-2}q_2(t,s,(t-s)^{-1/\gamma}x).
\end{align}

These kernels have uniform upper bounds.
\begin{lemma}       \label{ker bnd}
It holds that
\begin{align*}
\sup_{t>s,x}|\Delta^{\gamma/2} q_1(t, s, x)| < \infty,
\end{align*}
\begin{align*}
\sup_{t>s,x}|\frac{\partial}{\partial x^i}\Delta^{\gamma/2} q_1(t, s, x)| < \infty,
\end{align*}
and
\begin{align*}
\sup_{t>s,x}| q_2(t, s, x)| < \infty.
\end{align*}
\end{lemma}
\begin{proof}
From the properties of the Fourier transform, these are easy consequences of \eqref{A1}.
The lemma is proved.
\end{proof}

\begin{lemma}   \label{ker int fin}
Let $0<\delta < \Big(\frac{1}{2} \wedge \gamma \Big)$.
Then
\begin{align}           \label{lemma 531}
\sup_{t>s} \int_{\fR^d} \Big| |x|^{\frac{d}{2} +\delta }|\Delta^{\gamma/2} q_1(t, s, x)| \Big|^2~dx < \infty,
\end{align}
\begin{align}               \label{lemma 532}
\sup_{t>s} \int_{\fR^d} \Big| |x|^{\frac{d}{2}+\delta}|\frac{\partial}{\partial x^i}\Delta^{\gamma/2} q_1(t, s, x)| \Big|^2~dx < \infty,
\end{align}
and
\begin{align}               \label{lemma 533}
\sup_{t>s}\int_{\fR^d} \Big| |x|^{\frac{d}{2}+\delta}|q_2(t, s, x)| \Big|^2~dx < \infty.
\end{align}
\end{lemma}
\begin{proof}
First we prove \eqref{lemma 531}. Let $t>s$. By Parseval's identity,
\begin{align*}
&\int_{\fR^d} \Big||x|^{\frac{d}{2} + \delta  }|\Delta^{\gamma/2} q_1(t, s, x)| \Big|^2 ~dx  \\
&=N\int_{\fR^d} \Big|\Delta^{\frac{d}{4}-\lfloor \frac{d}{4} \rfloor
+\frac{\delta}{2}} \Delta^{\lfloor \frac{d}{4} \rfloor}\Big(|\xi|^\gamma  \exp \big(\int_s^t \psi(r,(t-s)^{-1/\gamma}\xi)~dr \big) \Big) \Big|^2 ~dx .
\end{align*}

We  apply Lemma \ref{frac est} with
$$
 \eta=d/2-2\lfloor \frac{d}{4} \rfloor +\delta, \quad \varsigma=\gamma-2\lfloor \frac{d}{4} \rfloor, \quad c=\nu
 $$
 and
 $$
 h(x)=\Delta^{\lfloor \frac{d}{4} \rfloor} \Big(|\xi|^\gamma  \exp \big(\int_s^t \psi(r,(t-s)^{-1/\gamma}\xi)~dr \big) \Big).
$$
Note that since $\gamma>\delta$, we have $\varsigma>\eta-d/2$.
Also, by \eqref{A1},
\begin{align*}
\Big| \Delta^{\lfloor \frac{d}{4} \rfloor} \Big(|\xi|^\gamma  \exp \big(\int_s^t \psi(r,(t-s)^{-1/\gamma}\xi)~dr \big) \Big) \Big|
\leq N |\xi|^{\gamma - 2\lfloor \frac{d}{4} \rfloor} e^{-\nu |\xi|^\gamma}.
\end{align*}
Thus (\ref{ck}) is satisfied with the above setting.

One can easily check that
$$
\eta\in \begin{cases} [0,1), & \text{if}\,\, d=4k, 4k+1\,\,\text{for some integer}\,k\\
[1,2), & \text{otherwise}
\end{cases}
$$
Therefore it is enough to prove (\ref{ddck}) if $d=4k$ or  $4k+1$ for some integer $k$ and  (\ref{ddck2}) for the other case. These are easy consequences of
\eqref{A1}, that is, we have
\begin{align*}
\Big| D^1 \Delta^{\lfloor \frac{d}{4} \rfloor} \Big(|\xi|^\gamma  \exp \big(\int_s^t \psi(r,(t-s)^{-1/\gamma}\xi)~dr \big) \Big) \Big|
\leq N |\xi|^{\gamma - 2\lfloor \frac{d}{4} \rfloor-1} e^{-\nu |\xi|^\gamma},
\end{align*}
and
\begin{align*}
\Big|D^2 \Delta^{\lfloor \frac{d}{4} \rfloor} \Big(|\xi|^\gamma  \exp \big(\int_s^t \psi(r,(t-s)^{-1/\gamma}\xi)~dr \big) \Big) \Big|
\leq N |\xi|^{\gamma - 2\lfloor \frac{d}{4} \rfloor-2} e^{-\nu |\xi|^\gamma}.
\end{align*}
Hence (\ref{lemma 531}) is proved.

Both \eqref{lemma 532} and \eqref{lemma 533} can be proved similarly. We only remark  main differences.
Due to \eqref{A1}, for any $i=1,2,3,\ldots,d$ and multi-index $|\beta| \leq \lfloor \frac{d}{2} \rfloor+1$
\begin{align*}
\Big| D^\beta \Big(|\xi|^\gamma \xi^i \exp \big(\int_s^t \psi(r,(t-s)^{-1/\gamma}\xi)~dr \big) \Big) \Big|
\leq N |\xi|^{\gamma +1 - |\beta|} e^{-\nu |\xi|^\gamma}
\end{align*}
and
\begin{align*}
&\Big| D^\beta \Big(\psi(s,(t-s)^{-1/\gamma} \xi)|\xi|^{\gamma} \exp \big(\int_s^t \psi(r,(t-s)^{-1/\gamma}\xi)~dr \big) \Big) \Big| \\
&\leq N (t-s)^{-1} |\xi|^{2\gamma - |\beta|} e^{-\nu |\xi|^\gamma}.
\end{align*}
Hence for \eqref{lemma 532} we  apply Lemma \ref{frac est} with
$$
 \eta=d/2-2\lfloor \frac{d}{4} \rfloor +\delta, \quad \varsigma=\gamma+1-2\lfloor \frac{d}{4} \rfloor, \quad c=\nu
 $$
 and
 $$
 h(x)=\Delta^{\lfloor \frac{d}{4} \rfloor} \Big(|\xi|^\gamma \xi^i \exp \big(\int_s^t \psi(r,(t-s)^{-1/\gamma}\xi)~dr \big) \Big).
$$
On the other hand for \eqref{lemma 533} we  apply Lemma \ref{frac est} with
$$
 \eta=d/2-2\lfloor \frac{d}{4} \rfloor +\delta, \quad \varsigma = 2\gamma-2\lfloor \frac{d}{4} \rfloor, \quad c=\nu
 $$
 and
 $$
 h(x)=\Delta^{\lfloor \frac{d}{4} \rfloor} \Big( (t-s)\psi(s,(t-s)^{-1/\gamma} \xi) |\xi|^{\gamma}  \exp \big(\int_s^t \psi(r,(t-s)^{-1/\gamma}\xi)~dr \big) \Big).
$$
We skip the  details. The lemma is proved.
\end{proof}

By making full use of above lemmas, we obtain kernel estimates for $\Delta^{\gamma/2}p(t,s, x)$.

\begin{lemma}           \label{lemma 54}
There exist constant $N>0$ and $\varepsilon\in (0,1)$ such that for all $s>r$,  and $c>0$
\begin{align*}
\int_r^s \int_{|z| \geq c} |\Delta^{\gamma/2}p(s,\tau, z)| ~dz d\tau \leq N(s-r)^{\varepsilon} c^{-\varepsilon \gamma}.
\end{align*}
\end{lemma}

\begin{proof}
From \eqref{rela p q1},
\begin{align*}
\int_{|z| \geq c} |\Delta^{\gamma/2}p(s,\tau, z)| ~dz
=(s-\tau)^{-1}\int_{(s-\tau)^{1/\gamma}|z| \geq c} |\Delta^{\gamma/2} q_1(s,\tau, z)| ~dz.
\end{align*}
For $0<\varepsilon <1$, if $(s-\tau)^{1/\gamma}|z| \geq c$, then
$$
(s-\tau)^{-1} \leq (s-\tau)^{-1+\varepsilon} \big(\frac{|z|}{c} \big)^{\varepsilon \gamma}.
$$
Therefore
\begin{align} \label{lemma54 suf}
\int_r^s \int_{|z| \geq c} |\Delta^{\gamma/2}p(s,\tau, z)| ~dz d\tau
\leq
c^{-\varepsilon \gamma}\int_r^s (s-\tau)^{-1+\varepsilon} \int_{\fR^d} |z|^{\varepsilon \gamma}|\Delta^{\gamma/2} q_1(s,\tau, z)| ~dz d\tau.
\end{align}
We claim
\begin{align*}
\sup_{s>\tau>0}\int_{\fR^d} |z|^{\varepsilon \gamma}|\Delta^{\gamma/2} q_1(s,\tau, z)| ~dz < \infty.
\end{align*}
By Lemma \ref{ker bnd} and H\"older's inequality,
\begin{align*}
& \int_{\fR^d} |z|^{\varepsilon \gamma}|\Delta^{\gamma/2} q_1(s,\tau, z)| ~dz  \\
&\leq \int_{|z| <1 } |z|^{\varepsilon \gamma}|\Delta^{\gamma/2} q_1(s,\tau, z)| ~dz+\int_{|z| \geq 1} |z|^{\varepsilon \gamma}|\Delta^{\gamma/2} q_1(s,\tau, z)| ~dz \\
&\leq N+N \Big( \int_{|z| \geq 1} |z|^{-d-\varepsilon \gamma }~dz \Big)^{1/2}
\Big( \int_{|z| \geq 1} \Big||z|^{\frac{d+3\varepsilon \gamma}{2}}|\Delta^{\gamma/2} q_1(s,\tau, z)| \Big|^2 ~dz  \Big)^{1/2}.
\end{align*}
Due to Lemma \ref{ker int fin} (i) with small $\varepsilon$ so that $\frac{3\varepsilon \gamma}{2} < \Big(\frac{1}{2} \wedge \gamma \Big)$,
\begin{align*}
\sup_{s > \tau} &\int_{\fR^d} \Big||z|^{\frac{d+3\varepsilon \gamma}{2}}|\Delta^{\gamma/2} q_1(s,\tau, z)| \Big|^2 ~dz < \infty.
\end{align*}
Therefore, the claim is proved. Going back to \eqref{lemma54 suf}, we conclude that
\begin{align*}
\int_r^s \int_{|z| \geq c} |\Delta^{\gamma/2}p(s,\tau, z)| ~dz d\tau
\leq
N c^{-\varepsilon \gamma}\int_r^s (s-\tau)^{-1+\varepsilon} d\tau
\leq c^{-\varepsilon \gamma} (s-r)^{\varepsilon}.
\end{align*}
The lemma is proved.
\end{proof}

\begin{lemma}           \label{lemma 55}
There exist constants $N>0$ and $\varepsilon\in (0,1)$ such that for all $s > r  >a$ and $h \in \fR^d$
\begin{align}       \label{lemma 551}
\int_{-\infty}^a \int_{\fR^d} \big|\Delta^{\gamma/2}p(s,\tau, z+h)-\Delta^{\gamma/2}p(s,\tau, z)\big| ~dz d\tau \leq N|h|(s - a)^{-1/\gamma}
\end{align}
and
\begin{align*}
\int_{-\infty}^a \int_{\fR^d} |\Delta^{\gamma/2}p(s,\tau, z)-  \Delta^{\gamma/2}p(r,\tau, z)| ~dz d\tau
\leq N(s-r)(r -a)^{-1}.
\end{align*}
\end{lemma}

\begin{proof}
First we show \eqref{lemma 551}. From \eqref{rela p q1},
\begin{align}   \label{lem55 scal}
\frac{\partial}{\partial x^i}\Delta^{\gamma/2}p(s,\tau, z)
=(s-\tau)^{-d/\gamma}(s-\tau)^{-1-1/\gamma}\frac{\partial}{\partial x^i} \Delta^{\gamma/2}q_1(s,\tau, (s-\tau)^{-1/\gamma}z).
\end{align}
Fix $0 < \delta < \Big( \frac{1}{2} \wedge \gamma \Big)$.
Then by H\"older's inequality, Lemmas \ref{ker bnd} and \ref{ker int fin},
\begin{align*}
&\sup_{s>\tau}\Big((s-\tau)^{-d/\gamma}\int_{\fR^d} \Big|\frac{\partial}{\partial x^i} \Delta^{\gamma/2}q_1(s,\tau, (s-\tau)^{-1/\gamma}z) \Big| dz \Big)\\
&=\sup_{s>\tau}\int_{\fR^d} \Big|\frac{\partial}{\partial x^i} \Delta^{\gamma/2}q_1(s,\tau, z) \Big| dz \\
&\leq N+\sup_{s>\tau}\Big[ \Big(\int_{|z| \geq 1} |z|^{-d-2\delta}~dz\Big)^{1/2}
\Big(\int_{|z| \geq 1} \Big||z|^{d/2+\delta}\frac{\partial}{\partial x^i} \Delta^{\gamma/2}q_1(s,\tau, z) \Big|^2 dz \Big)^{1/2} \Big] \\
&< \infty.
\end{align*}
Therefore, by the mean-value theorem and \eqref{lem55 scal}
\begin{align*}
&\int_{-\infty}^a \int_{\fR^d} \big| \Delta^{\gamma/2}p(s,\tau, z+h)-\Delta^{\gamma/2}p(s,\tau, z) \big| ~dz d\tau  \\
&\leq |h|\int_{-\infty}^a \int_{\fR^d} |\nabla \Delta^{\gamma/2}p(s,\tau, z)| ~dz d\tau  \\
&\leq |h|\int_{-\infty}^a (s-\tau)^{-1-1/\gamma} \int_{\fR^d} \Big|\frac{\partial}{\partial x^i} \Delta^{\gamma/2}q_1(s,\tau, z) \Big|~dz d\tau \\
&\leq N|h|\int_{-\infty}^a (s-\tau)^{-1-1/\gamma} d\tau \leq N|h|\int_{s-a}^\infty \tau^{-1-1/\gamma} d\tau=N|h|(s - a)^{-1/\gamma}.
\end{align*}

In order to prove the second assertion, observe that by the mean-value theorem and \eqref{rela p q2},
\begin{align*}
&|\Delta^{\gamma/2}p(s,\tau, z)-  \Delta^{\gamma/2}p(r,\tau, z)| \\
&\leq |s-r||\frac{\partial}{\partial t}\Delta^{\gamma/2}p(\theta s + (1-\theta)r ,\tau, z)| \\
&\leq |s-r|(\theta s + (1-\theta)r-\tau)^{-d/\gamma -2}|q_2(\theta s + (1-\theta)r ,\tau, (\theta s + (1-\theta)r-\tau)^{-1/\gamma}z)|.
\end{align*}
Following the proof of the first assertion with Lemma \ref{ker int fin} (iii), we get
$$
\sup_{s > \tau, r >\tau, 0\leq \theta \leq 1}\int_{\fR^d} | q_2 (\theta s + (1-\theta)r ,\tau, z)|~dz <\infty.
$$
Therefore,
\begin{align*}
\int_{-\infty}^a \int_{\fR^d} |\Delta^{\gamma/2}p(s,\tau, z)-  \Delta^{\gamma/2}p(r,\tau, z)| ~dz d\tau
&\leq \int_{-\infty}^a \frac{|s-r|}{\big(\theta s + (1-\theta)r -\tau \big)^2} d\tau\\
&\leq |s-r|(r-a)^{-1}.
\end{align*}
The lemma is proved.
\end{proof}

{\bf Proof of Theorem \ref{pseudo thm}}

From Lemma \ref{lemma 54} and Lemma \ref{lemma 55}, it is proved that the kernel
$K(s,\tau,z):=\Delta^{\gamma/2}p(s,\tau, z)$
satisfies Assumption \ref{as 2}. Moreover, by the definition of the kernel,
\begin{align*}
\Big|\cF \Big(\Delta^{\gamma/2}p(t,s, \cdot) \Big)(\xi)\Big|
&\leq |\xi|^\gamma \Big|\exp \big(\int_s^{t} \psi(r,\xi)dr) \Big|  \\
&\leq  |\xi|^\gamma \exp \big( -\nu (t-s)|\xi|^\gamma )
\end{align*}
where the second inequality is due to \eqref{A1}. Hence $K(s,\tau,z)=\Delta^{\gamma/2}p(s,\tau, z)$ also satisfies assumption \ref{as 1} because obviously
$$
\sup_{\xi} \int_0^\infty |\xi|^\gamma \exp \big( -\nu t|\xi|^\gamma )~dt <\infty.
$$
Therefore, due to Theorem \ref{BMO theorem}, for any $p\geq 2$ it holds that
\begin{align}
                \label{befext}
\|\cG f\|_p\leq N\|f\|_p, \quad \forall f \in C_c^\infty(\fR^{d+1}).
\end{align}
Since the operator $f \to \cG f$ is linear and \eqref{befext} holds for all $f \in C_c^\infty(\fR^{d+1})$, the operator $\cG$ is extendible to a bounded linear operator  on $L_p(\fR^{d+1})$, and  \eqref{befext} holds for all $f \in L_p(\fR^{d+1})$.

Now assume $u\in C^{\infty}_c(\fR^{d+1})$. Denote $f:=u_t-A(t)u$. Then obviously $f\in L_p(\fR^{d+1})$. Thus to prove (\ref{7.9.1}) we only need to show $(-\Delta)^{\gamma/2}u=\cG f$.
Taking the Fourier transform to the equation $u_t-A(t)u=f$, one easily gets
$$
\hat{u}(t,\xi)=\int^t_{-\infty}e^{\int^t_s \psi(r,\xi)\,dr}\hat{f}(s,\xi)\,ds.
$$
This and the inverse Fourier transform certainly lead to
\begin{align}
                \label{sol equal}
u(t,x)=\int^t_{-\infty}p(t,s,\cdot)\ast f(s,\cdot)(x)ds, \quad (-\Delta)^{\gamma/2}u=\cG f.
\end{align}
These equalities are  because $f$ has compact support and is sufficiently smooth with respect to $x$ uniformly in $t$ (cf. Remark \ref{well define g}).

Next we
prove (\ref{7.9.1}) for $p\in (1,2)$ by using the duality argument. Let $q\in (2,\infty)$ be the conjugate of $p$. Consider the kernel
\begin{align*}
P(t,s,x)=K(-s,-t,x)
&=I_{-t<-s}\mathcal{F}^{-1}\left\{|\xi|^{\gamma}\exp\left(\int_{-t}^{-s}\psi(r,\xi)dr\right)\right\} \\
&=I_{t>s}\mathcal{F}^{-1}\left\{|\xi|^{\gamma}\exp\left(\int_{s}^{t}\psi(-r,\xi)dr\right)\right\}.
\end{align*}
Note that $\psi(-t,\xi)$ also satisfies Assumption \ref{as 7.8}. Define operator $\mathcal{P}$ by
$$
\mathcal{P}h(t,x):=\int_{\fR^{d+1}}P(t,s,x-y)h(s,y)dyds.
$$
Considering the change of variable $(s,t)\to(-s,-t)$,
we observe that by Fubini's theorem, for $f,g\in C_{c}^{\infty}(\fR^{d+1})$,
\begin{align*}
\int_{\fR^{d+1}}&g(s,y)\mathcal{G}f(s,y)~dyds
\\&=\int_{\fR^{d+1}}g(s,y)\left(\int_{\fR^{d+1}}K(s,t,y-x)f(t,x)~dxdt\right)dyds
\\&=\int_{\fR^{d+1}}f(-t,-x)\left(\int_{\fR^{d+1}}K(-s,-t,y)g(-s,y-x)~dyds\right)dxdt
\\&=\int_{\fR^{d+1}}f(-t,-x)\left(\int_{\fR^{d+1}}P(t,s,y)\tilde g(s,x-y)~dyds\right)dxdt
\\&=\int_{\fR^{d+1}}f(-t,-x)\mathcal{P}\tilde g(t,x)~dxdt
\end{align*}
where $\tilde g(t,x)=g(-t,-x)$. Then by H\"{o}lder inequality and the fact that $2<q<\infty$, we have
$$
\left|\int_{\fR^{d+1}}g(t,x)\mathcal{G}f(t,x)dxdt\right| \leq N\|f\|_{L_{p}}\|\mathcal{P} \tilde g\|_{L_{q}} \leq N\|f\|_{L_{p}}\|g\|_{L_{q}}.
$$
Since $g\in C^{\infty}_c(\fR^{d+1})$ is arbitrary, \eqref{befext}  is proved for $p\in (1,2)$.
Reminding \eqref{sol equal}, we obtain (\ref{7.9.1}) for $p \in (1,2)$. The theorem is proved.

\mysection{Applications}
              \label{appl}

For applications of Theorem \ref{pseudo thm} we introduce $2m$-order operator
$$
A_1(t)u:=(-1)^{m-1}\sum_{|\alpha|=|\beta|=m}a^{\alpha \beta}(t)D^{\alpha+\beta}u,
$$
and $\gamma$-order (nonlocal) operator
$$
 \quad A_2(t):=-a(t)(-\Delta)^{\gamma/2},
$$
where the coefficients $a^{\alpha \beta}(t)$ and $a(t)$ are bounded complex-valued measurable functions satisfying
$$\nu<\Re[a(t)]<\nu^{-1},
$$
and
\begin{align*}
\nu |\xi|^{2m} \leq  \sum_{|\alpha|=|\beta|=m}  \xi^\alpha \xi^\beta \Re\left[a^{\alpha \beta}(t)\right]  \leq \nu^{-1} |\xi|^{2m}, \quad \forall \xi\in \fR^d.
\end{align*}

\begin{corollary}         \label{maintheorem3}
Let $p > 1$. Then for any $u \in C_c^\infty(\fR^{d+1})$,
$$
\|u_t\|_{L_p(\fR^{d+1})}+\| (-\Delta)^{m}u\|_{L_p(\fR^{d+1})} \leq N \|u_t-A_1(t)u\|_{L_p(\fR^{d+1})},
$$
where $N$ depends only on $p,\nu, m$ and $d$.
\end{corollary}

\begin{proof}
It is obvious that the symbol
$\psi(t,\xi)=-a^{\alpha \beta}(t) \xi^\alpha \xi^\beta$
satisfies  (\ref{A1}) with $\gamma=2m$ and any multi-index $\alpha$. Thus the corollary follows from Theorem \ref{pseudo thm}.
\end{proof}

\begin{corollary}         \label{maintheorem2}
Let $p > 1$.  Then for any $u \in C_c^\infty(\fR^{d+1})$,
$$
\|u_t\|_{L_p(\fR^{d+1})}+\| (-\Delta)^{\gamma/2}u\|_{L_p(\fR^{d+1})} \leq N \|u_t-A_2(t)u\|_{L_p(\fR^{d+1})},
$$
where $N$ depends only on $p,\nu, \gamma$ and $d$.
\end{corollary}

\begin{proof}
The symbol related to the operator $A_2(t)$ is  $-a(t)|\xi|^{\gamma}$, and therefore the corollary follows from  Theorem \ref{pseudo thm}.
\end{proof}
Recall we defined $(-\Delta)^{\gamma/2}$ as the operator with symbol $|\xi|^{\gamma}$ for any $\gamma\in (0,\infty)$.
For further applications of Theorem \ref{pseudo thm}, we consider a product of $(-\Delta)^k$ and an integro-differential operator $\cL_0=\cL_{0,\gamma}$. We remark that in place of $(-\Delta)^k$ one can consider many other pseudo-differential or high order differential operators.

\vspace{2mm}

Fix $\gamma\in (0,2)$, and for $k=0,1,2,\cdots$ denote
\begin{align*}
 &\cL_{k}(t)u =(-\Delta)^k\cL_{0,\gamma} u\\
&:=\int_{\fR^d \setminus \{0\}} \Big( (-\Delta)^k u(t,x+y)- (-\Delta)^k u(t,x)-\chi(y)(\nabla (-\Delta)^k u (t,x),y) \Big)
\frac{m(t,y) }{|y|^{d+\gamma}}dy
\end{align*}
where $\chi(y)= I_{\gamma >1} + I_{|y|\leq1}I_{\gamma=1}$  and $m(t,y)\geq 0$ is a  measurable function satisfying  the following conditions :

(i) If $\gamma=1$ then
\begin{align}           \label{cancel}
\int_{\partial B_1} w m(t,w)~S_1(dw)=0, \quad \forall t >0,
\end{align}
where  $\partial B_{1}$ is the unit sphere in $\fR^d$ and $S_{1}(dw)$ is the surface measure on it.

(ii) The function $m=m(t,y) $ is zero-order homogeneous and differentiable in $y$ up to $d_0 = \lfloor\frac{d}{2}\rfloor+1$.

(iii) There is a constant $K$ such that for each  $t \in \fR$
$$
\sup_{|\alpha| \leq d_0, |y|=1} |D^\alpha_y m^{(\alpha)} (t,y) | \leq K.
$$
It turns out that the operator $\cL_k$ is a pseudo differential operator with  symbol
\begin{align*}
\psi(t,\xi) = - c_1 |\xi|^{2k}\int_{\partial B_{1}} |(w,\xi)|^\gamma [1-i\varphi^{(\gamma)}(w,\xi)] m(t,w)~S_{1}(dw),
\end{align*}
\begin{align*}
\varphi^{(\gamma)}(w,\xi)=c_2\frac{(w,\xi)}{|(w,\xi)|} I_{\gamma \neq 1}- \frac{2}{\pi} \frac{(w,\xi)}{|(w,\xi)|} \ln |(w,\xi)| I_{\gamma=1},
\end{align*}
and $c_1(\gamma,d)$, $c_2(\gamma,d)$ are certain positive constants.

(iv) There is a constant $N_{0}>0$ such that the symbol $\psi(t,\xi)$ of $\cL_k$ satisfies
\begin{equation}
             \label{eqn psi7}
\sup_{t,|\xi|=1} \Re[\psi(t,\xi)] \leq  -N_{0}.
\end{equation}
One can check that  (\ref{eqn psi7}) holds if there exists a constant $c>0$ so that $m(t,y)>c$ on a set  $E\subset \partial B_1$ of positive  $S_1(dw)$-measure.

\begin{corollary}         \label{maintheorem4}
Let $p> 1$ and $k=0,1,2,\cdots$. Then under above conditions (i)-(iv) on $m(t,y)$ it holds that  for any $u \in C_c^\infty(\fR^{d+1})$
$$
\|u_t\|_{L_p(\fR^{d+1})}+\| (-\Delta)^{\gamma/2+k}u\|_{L_p(\fR^{d+1})} \leq N \|u_t-\cL_k u\|_{L_p(\fR^{d+1})},
$$
where $N$ depends only on $p, \gamma, k,d, N_0$ and $K$.
\end{corollary}

\begin{proof}

 Note that for $\xi \neq 0$
 \begin{align*}
\psi(t, \xi)&
=|\xi|^{2k+\gamma}\psi\Big(t, \frac{\xi}{|\xi|} \Big)=: |\xi|^{2k+\gamma} \tilde \psi(t,\xi).
 \end{align*}
The above equality  is obvious if $\gamma \neq 1$, and if  $\gamma =1$ then by \eqref{cancel}
\begin{align*}
\psi(t, \xi)
&=|\xi|^{2k+1} \psi\Big(t, \frac{\xi}{|\xi|} \Big) + |\xi|^{2k}\ln |\xi| \int_{\partial B_1} (w,\xi) m(t,w)~S_1(dw) \\
&=|\xi|^{2k+1} \psi\Big(t, \frac{\xi}{|\xi|} \Big).
\end{align*}
By using condition (iii) one can check (see e.g. \cite[Remark 2.6]{MP92}) that for
 any multi-index $\alpha$, $|\alpha|\leq d_0$,  there exists a  constant $N=N(\alpha)$ such that
$$
|D^\alpha \tilde \psi(t,\xi)| \leq N|\xi|^{-|\alpha|}.
$$
Thus it is obvious that the given symbol $\psi$ satisfies (\ref{A1}). The corollary is proved.

\end{proof}

Next we discuss the issue regarding the compositions and powers of operators.
Let $B_1(t)$ and $B_2(t)$ be linear operators with symbols $\psi_1(t)$ and  $\psi_2(t)$ satisfying (\ref{A1}), that is there exist constants  $\gamma_1, \gamma_2, \nu_1, \nu_2>0$ so that
$$
      \Re [-\psi_i(t,\xi)]  \geq  \nu_i |\xi|^{\gamma_i},\quad
 |D^\alpha \psi_i(t,\xi)| \leq \nu_i^{-1} |\xi|^{\gamma_i -|\alpha|}, \quad (i=1,2),
 $$
 for any multi-index $\alpha$, $|\alpha|\leq \lfloor \frac{d}{2} \rfloor+1$. Fix $a,b>0$, and denote $\gamma:=a\gamma_1+b\gamma_2$. Consider $\gamma$-order  operator
 $$
 C(t)=-(-A_1(t))^a(-A_2(t))^b
 $$
 with the symbol $\psi=-(-\psi_1)^a(-\psi_2)^b$. It is easy to check that there exists a constant $N>0$ so that  for any multi-index $\alpha$, $|\alpha|\leq \lfloor \frac{d}{2} \rfloor+1$,
  $$
 |D^{\alpha} \psi (t,\xi)| \leq N |\psi|^{\gamma-\alpha}, \quad \quad  \xi \in \fR^d\setminus\{0\}.
$$
Therefore, Theorem \ref{pseudo thm} is applicable to the operator $C(t)=-(A_1(t))^a(A_2(t))^b$ if
 \begin{equation}
            \label{eqn 7.19}
  \Re [-\psi(t,\xi)]=\Re[(-\psi_1)^a(-\psi_2)^b]\geq N^{-1} |\xi|^{\gamma}, \quad \forall \, \xi\in \fR^d.
 \end{equation}
 Obviously (\ref{eqn 7.19}) is satisfied if, for instance, the symbols $\psi_i(t,\xi)$ are real-valued. In this case,  for any $u\in C^{\infty}_c(\fR^{d+1})$, we have
 $$
 \|u_t\|_{L_p(\fR^{d+1})}+\|(-\Delta)^{\gamma/2}u\|_{L_p(\fR^{d+1})}\leq N\|u_t-C(t)u\|_{L_p(\fR^{d+1})}.
 $$

\end{document}